\documentclass[12 pt,english,oneside]{amsart}
\usepackage[top=1.16in, bottom=1.16in, left=1.23in, right=1.23in]{geometry}
\usepackage[dvips]{graphicx}
\usepackage{amsmath}
\usepackage[latin1]{inputenc}
\usepackage{amsfonts}
\usepackage{amssymb}
\usepackage{graphicx,epsfig}
\usepackage{amsthm}

\usepackage{amscd}
\usepackage[all]{xy}
\usepackage{fancyhdr}
\usepackage[T1]{fontenc}

\DeclareMathOperator{\id}{id}

\DeclareMathOperator{\en}{End}

\newcommand*{\qqs}{\ensuremath{\Join}}

\newcommand*{\s}{\ensuremath{\sigma}}

\newcommand*{\ot}{\ensuremath{\otimes}}

\newcommand*{\ve}{\varepsilon}

\newcommand*{\uot}{\ensuremath{\underline{\otimes}}}

\newcommand*{\gf}{\mathbb{K}}

\newcommand*{\sym}{\mathfrak{S}}

\newcommand*{\sq}{\square}

\newcommand*{\cyd}{{}^H_H\mathcal{YD}}

\newcommand*{\chb}{{}^H_H\mathcal{M}_H^H}

\newcommand*{\ts}{\otimes}
\newcommand*{\ra}{\rightarrow}
\newcommand*{\vp}{\varphi}
\newcommand*{\wt}{\widetilde}

\newtheorem{theorem}{Theorem}[section]
\newtheorem{lemma}[theorem]{Lemma}
\newtheorem{proposition}[theorem]{Proposition}
\newtheorem{corollary}[theorem]{Corollary}

\theoremstyle{definition}
\newtheorem{definition}[theorem]{Definition}
\newtheorem{example}[theorem]{Example}

\theoremstyle{remark}
\newtheorem{remark}[theorem]{Remark}

\allowdisplaybreaks

\numberwithin{equation}{section}

\begin{document}

\title{Cofree Hopf algebras on Hopf bimodule algebras}

\author{Xin Fang}
\address{Universit\'e Paris Diderot-Paris VII, Institut de Math\'ematiques de Jussieu - Paris Rive Gauche, B\^atiment Sophie Germain, Case 7012, 75205 Paris Cedex 13, France.}
\email{xinfang.math@gmail.com}

\author{Run-Qiang Jian}
\address{School of Computer Science, Dongguan University
of Technology, 1, Daxue Road, Songshan Lake, 523808, Dongguan, P.
R. China}
\email{jian.math@gmail.com}
\thanks{}


\date{}



\begin{abstract}
We investigate a Hopf algebra structure on the cotensor coalgebra
associated to a Hopf bimodule algebra which contains universal
version of Clifford algebras and quantum groups as examples. It is
shown to be the bosonization of the quantum quasi-shuffle algebra
built on the space of its right coinvariants. The universal property and a Rota-Baxter algebra structure are established on this new algebra.
\end{abstract}

\maketitle

\section{Introduction}
The notion of a cotensor coalgebra was first introduced by Nichols
(\cite{N}) more than three decades ago. It is the dual version of
tensor algebras built on bimodules over algebras. By using the
universal property of cotensor coalgebras, Nichols constructed an
algebra structure on the underlying space of the cotensor
coalgebra. He investigated the subalgebra generated by elements of
degree 0 and 1, which is called the Nichols algebra or quantum shuffle algebra nowadays. This
sort of algebras lead to the quantum shuffle interpretation of
quantum groups (\cite{Ro}) and a series of works on the
classification of finite dimensional pointed Hopf algebras over
finite abelian groups of order prime to $210$ (cf. \cite{AS1},
\cite{AS2} and the references therein). These unexpected relations
with other mathematical objects such as operads and logarithm conformal field theory evoke great
interests on cotensor coalgebras. 
\par
Nevertheless, the study of this
subject is just at the beginning. On one hand, the Nichols algebra
itself admits various algebraic structures and many open questions. One may
expect that new properties on these objects about this subject and
connections with other algebraic or geometric objects will be discovered. On
the other hand, we could generalize the Nichols algebra in a much
bigger framework. For the second direction, the first step was
advised by Rosso. He considered a Hopf bimodule $M$ with an
associative product $m$. After imposing some compatibilities
between the multiplication and the Hopf bimodule structure, he
constructed a Hopf algebra structure on the cotensor coalgebra of
this Hopf bimodule algebra $(M,m)$. We call this new algebra the
cofree Hopf algebra on the Hopf bimodule algebra.
\par
 As in the
case of quantum shuffle algebras, these Hopf algebras bring
surprising applications to quantum groups. Recently, Rosso and the
first author showed that the whole quantum group associated to  a symmetrizable Kac-Moody Lie algebra can be realized as a natural
quotient of the subalgebra of the cofree Hopf algebra built on
some well-chosen Hopf bimodule algebras generated by elements of degree $0$
and $1$ (\cite{FR}). This provides a "global" construction
of the whole quantum group, not just the positive part. We expect that this result could give some new insights on quantum groups. Moreover, Rosso's construction is
entirely generalized in the framework of multi-brace cotensor Hopf
algebras and the algebra structures on cotensor coalgebras are
classified there in spirit of Loday and Ronco \cite{LR}.
\par
This new algebra built on a Hopf bimodule algebra $M$ has a richer combinatorial structure and therefore is complicated. The explicit formula of this new multiplication involves all
structures on $M$ and our first task is to understand it. Analogue to the relation between cofree
Hopf algebras on Hopf bimodules and quantum shuffle algebras
(\cite{Ro}) the quantum quasi-shuffle algebra \cite{JRZ} can be found inside of a cofree Hopf algebra on a Hopf bimodule algebra. More
precisely, the new algebra is isomorphic, as a Hopf algebra, to
the bosonization of quantum quasi-shuffle. These results should be known by Rosso, but he did not published any proof.
Due to the importance of this algebra, it seems quite necessary to
write down a complete and self-contained demonstration of this observation. 
\par
In this
paper, after a careful analysis, we use a universal property
showed in \cite{JRZ} to provide a detailed proof.
Thanks to this isomorphism, we can concentrate our study on
quantum quasi-shuffle algebras, which are
the quantization of the classical quasi-shuffle algebras
introduced by Newman and Radford \cite{NR}, Hoffmann (\cite{H}),
Guo and Keigher (\cite{GK}) independently. They are constructed as
a special case of quantum multi-brace algebras (\cite{JR}) and
have some interesting properties and applications to other
mathematical objects (cf. \cite{J}, \cite{J2}, \cite{JRZ}). We can
extend the results about quantum quasi-shuffle algebras to the
cofree Hopf algebras on Hopf bimodule algebras. These extensions are
significant for the reason that the former ones are not real Hopf
algebras but the latter ones are. For example, the second author
constructed in \cite{J} a Rota-Baxter algebra structure on quantum
quasi-shuffle algebras, which can
be extended to the bosonized Hopf algebra and then gives a Rota-Baxter
algebra structure on the cofree Hopf algebra built on  Hopf
bimodule algebras.

We also establish some important tools for the further study of
this new algebra, such as universal property. Finally, we provide
examples of our new algebras, such as universal Clifford algebras
and universal quantum groups. These examples demonstrate that
the cofree Hopf algebras on Hopf bimodule algebras include many
important algebras as special case. As we mentioned above that
these new objects admit plenty of structures, this paper can be viewed
as an opening of this subject.

This paper is organized as follows. We provide a detailed study of
the algebra structure built on the cofree Hopf algebras on Hopf
bimodule algebras in Section 2. In Section 3, a
Rota-Baxter algebra structure on the new algebra, as well as a
universal property, are established. Several examples including
universal Clifford algebras and universal quantum groups are provided in
Section 4.

\noindent\textbf{Notations.} In this paper, we denote by $\gf$ a
ground field of characteristic 0. All vector spaces, algebras,
coalgebras and tensor products we discuss are defined over $\gf$
if not specified otherwise. For a vector space $V$, we denote by
$T(V)$ the tensor vector space of $V$, by $\ot$ the tensor product
within $T(V)$, and by $\uot$ the one between $T(V)$ and $T(V)$.

A braiding $\s$ on a vector space $V$ is an invertible linear map
in $\en(V\ot  V)$ satisfying the quantum Yang-Baxter equation on
$V^{\ot  3}$: $$(\s \ot \id_{V})(\id_{V}\ot \s )(\s \ot
\id_{V})=(\id_{V}\ot  \s )(\s \ot \id_{V})(\id_{V}\ot \s ).$$ A
braided vector space $(V,\s )$ is a vector space $V$ equipped with
a braiding $\s $. For any $n\in \mathbb{N}$ and $1\leq i\leq n-1$,
we denote by $\s _i$ the operator $\id_V^{\ot  i-1}\ot \s \ot
\id_V^{\ot  n-i-1}\in \en(V^{\ot  n})$. We denote by $\sym_{n}$
the symmetric group acting on the set $\{1,2,\ldots,n\}$ and by
$s_{i}$, $1\leq i\leq n-1$, the standard generators of $\sym_{n}$
permuting $i$ and $i+1$. For any $w\in \sym_{n}$, we denote by
$T^\s _w$ the corresponding lift of $w$ in the braid group
$\mathfrak{B}_n$, defined as follows: if $w=s_{i_1}\cdots s_{i_l}$
is any reduced expression of $w$, then $T^\s _w=\s _{i_1}\cdots \s
_{i_l}$. By Theorem 4.12 in \cite{KT}, it is well-defined.

We define $\beta:T(V)\uot T(V)\rightarrow T(V)\uot T(V)$ by
requiring that the restriction of $\beta$ on $V^{\ot i}\uot V^{\ot
j}$, denoted by $\beta_{ij}$, is $T^\s _{\chi_{ij}}$, where
\[\chi_{ij}=\left(\begin{array}{cccccccc}
1&2&\cdots&i&i+1&i+2&\cdots & i+j\\
j+1&j+2&\cdots&j+i&1& 2 &\cdots & j
\end{array}\right)\in \sym_{i+j},\] for any $i,j\geq 1$. For convenience, we denote by
$\beta_{0i}$ and $\beta_{i0}$ the usual flip.

Let $(H,\Delta,\varepsilon, S)$ be a Hopf algebra. We denote
$\Delta^{(0)}=\id$, $\Delta^{(1)}=\Delta$, and
$\Delta^{(n+1)}=(\Delta^{(n)}\ot \id)\Delta$ recursively for
$n\geq 1$. The coradical $\mathrm{corad}(H)$ of $H$ is the sum of
all simple subcoalgebras of $H$ (cf. \cite{S}).

In the sequel, we adopt Sweedler's notation for coproducts and
comodule structure maps. For any $h\in H$,
$$\Delta(h)=\sum h_{(1)}\ot  h_{(2)}.$$For a left $H$-comodule $(M,\delta_L)$ (resp. right $H$-comodule $(M,\delta_R)$) and any $ m\in M$,
$$\delta_L(m)=\sum m_{(-1)}\ot  m_{(0)}\ \ \  (\text{resp.}\ \ \delta_R(m)=\sum m_{(0)}\ot  m_{(1)}).$$

\section{Cofree Hopf algebras on Hopf bimodule algebras}
In this section, we always assume
$(H,\mu_H,\Delta_H,\varepsilon_H,S)$ is a Hopf algebra.

\subsection{Algebra objects in $\cyd$ and $\chb$}
We start by recalling the equivalence of $H$-Yetter-Drinfeld module categories and $H$-Hopf bimodule categories. Our basic references are \cite{AS1}, \cite{Ro} and \cite{S}.

\subsubsection{Yetter-Drinfeld module algebras}

\begin{definition}A \emph{(left) $H$-Yetter-Drinfeld
module} is a vector space $V$ equipped simultaneously with a left
$H$-module structure $\cdot$ and a left $H$-comodule structure
such that whenever $h\in H$ and $v\in V$,
$$\sum h_{(1)}v_{(-1)}\otimes h_{(2)}\cdot v_{(0)}=\sum
(h_{(1)}\cdot v)_{(-1)}h_{(2)}\otimes (h_{(1)}\cdot v)_{(0)}.
$$\end{definition}

The category of $H$-Yetter-Drinfeld modules, denoted by $\cyd$,
consists of the following data:
\begin{enumerate}
\item objects in $\cyd$ are $H$-Yetter-Drinfeld modules;
\item morphisms in $\cyd$ are linear maps which are both module and
comodule morphisms.
\end{enumerate}

Together with the usual tensor product, $\cyd$ is a tensor
category. Moreover it is braided with the following commutativity
constraint: for any $V,W\in\cyd$,
\[
\s:V\ot W \rightarrow W\ot V,\ \
 v\ot w\mapsto\sum v_{(-1)}\cdot w\ot v_{(0)}.
\]

A unital algebra $(V,m,u)$ in $\cyd$ is a unital associative
algebra such that the underlying space $V$ is an object of $\cyd$
and its multiplication $m$ and unit $u$ are morphisms in $\cyd$.

\subsubsection{Hopf bimodule algebras}

\begin{definition}
An \emph{$H$-Hopf bimodule} is a vector space $M$ equipped simultaneously with $H$-bimodule and $H$-bicomodule structures (with left coaction $\delta_L$ and right
coaction $\delta_R$) such that both $\delta_L$ and $\delta_R$ are $H$-bimodule maps.
\end{definition}

The category of $H$-Hopf bimodules, denoted by $\chb$, consists of
the following data:\begin{enumerate}
\item objects in $\chb$ are $H$-Hopf bimodules;
\item morphisms in $\chb$ are linear maps which are both $H$-bimodule and $H$-bicomodule morphisms.
\end{enumerate}

Taking tensor product over $H$, $\chb$ is a tensor category: the bimodule structure is on two extremal terms and the bicomodule structure arises from the tensor product. In his seminal paper \cite{W}, Woronowicz showed that $\chb$ is
braided.

As in the case of Yetter-Drinfeld modules, we can consider unital algebras in
$\chb$. More precisely, we call $M\in \chb$ an \emph{$H$-Hopf
bimodule algebra} if $M$ is an associative algebra with
multiplication
$$\mu:M\ot M\rightarrow M\in {}_H^H\mathcal{M}_H^H,$$
where the tensor product is the one in ${}_H^H\mathcal{M}_H^H$.

\subsubsection{Equivalence of categories}

\begin{theorem}[\cite{Ro}]
There exists an equivalence of braided tensor categories
$\chb\sim\cyd$ sending an $H$-Hopf bimodule $M$ to the set of its
right coinvariants $M^R=\{m\in M|\ \delta_R(m)=m\ot 1_H\}$.
\end{theorem}

We give some details of this equivalence. By Theorem 4.1.1 in
\cite{S}, $M^R=P_R(M)$, where $P_R$ is the projection from $M$ onto $M^R$ given by $P_R(m)=\sum m_{(0)}S(m_{(1)})$ for $m\in M$. A direct verification shows that $M^R$ inherits a left $H$-comodule structure from $(M,\delta_L)$
and a left $H$-module structure given by the adjoint action:
$$h\cdot m=\sum h_{(1)}mS(h_{(2)}),\ \ \text{for}\ h\in H\
\text{and}\ m\in M^R.
$$
With these structures, $(M^R,\cdot, \delta_L)$ is an
$H$-Yetter-Drinfeld module.

Conversely, if $V$ is an $H$-Yetter-Drinfeld module, we denote
$M=V\ot H$ and define: for any $h,h'$ and $v\in V$,
$$
h' (v\ot h)=\sum h'_{(1)}\cdot v\ot h'_{(2)}h,\ \  (v\ot h)h' =v\ot hh',
$$
$$\delta_L(v\ot h)=\sum v_{(-1)}h_{(1)}\ot(v_{(0)}\ot h_{(2)}),\ \
\delta_R(v\ot h)=\sum (v\ot h_{(1)})\ot
h_{(2)}.
$$
Then $M$, equipped with these structures, is an $H$-Hopf bimodule.

The algebras in $\cyd$ and those in $\chb$ are corresponded under
the category equivalence. For an $H$-Hopf bimodule algebra $M$,
$M^R$ is naturally an algebra object in $\cyd$. In order to
construct $H$-Hopf bimodule algebras from algebras in $\cyd$, we
need the notion of smash product. Suppose that $V$ is an algebra
in $\cyd$. The smash product $V\#H$ of $V$ and $H$ is defined to
be $V\ot H$ as a vector space and
$$(v\# h)(v'\# h')=\sum v(h_{(1)}\cdot v')\# h_{(2)}h',\ \
v,v'\in V, h,h'\in H,$$where we use $\#$ instead of $\ts$ to
emphasize the smash product. One can prove that $V\#H$, together
with the above $H$-Hopf bimodule structure, is an $H$-Hopf
bimodule algebra.

\begin{example}
It is well-known that $H$ is a Yetter-Drinfeld module over itself
with the following structures: for any $x,h\in H$,
\[
x\cdot h=\sum x_{(1)}h S(x_{(2)}),\ \ \rho(h)=\sum h_{(1)}\ot h_{(2)}.
\]
In addition, it is a unital algebra in $\cyd$. Then $H\ot H$ is an
$H$-Hopf bimodule algebra with $H$-Hopf bimodule structure defined
before and algebra structure $$(h^1\ot h^2)(h^3\ot h^4)=\sum
h^1h^2_{(1)}h^3S(h^2_{(2)})\ot h^2_{(3)}h^4.$$
\end{example}

\subsection{Cofree Hopf algebras on $H$-Hopf bimodule algebras}
\subsubsection{Cotensor coalgebras}
Given two $H$-Hopf bimodules $M$ and $N$, we denote by $M\square
_H N$ their cotensor product, i.e., the kernel of the map
$\id_M\ot \delta_L-\delta_R\ot \id_N:M\ot N\rightarrow M\ot H\ot
N$. It is a subspace of $M\ot  N$. We endow $M\square _H N$ with
the following $H$-Hopf bimodule structure:
$$h(m\ot n)=\sum h_{(1)}m\ot h_{(2)}n,\ \
(m\ot n)h=\sum mh_{(1)}\ot nh_{(2)},
$$
$$
\rho_L(m\ot n)=\sum m_{(-1)}\ot (m_{(0)}\ot n),\ \
\rho_R(m\ot n)=\sum (m\ot n_{(0)})\ot n_{(1)}.
$$
The cotensor product is associative, which allows one to define
recursively $M^{\sq _H 0}=H$, $M^{\sq _H 1}=M$, and $M^{\sq _H
n+1}=M^{\sq _H n}\sq_H M$ for $n\geq 1$.  We denote
$T^c_H(M)=\bigoplus_{n\geq 0}M^{\sq _H n}$ and use $\sum m^1\sq
\cdots \sq m^n$ instead of $\sum m^1\ot \cdots \ot m^n$ to
indicate an element in $M^{\sq _H n}$. We define a linear map
$\Delta:T^c_H(M)\rightarrow T^c_H(M)\uot T^c_H(M)$ as follows: for
any $h\in H$,
$$\Delta(h)=\Delta_H(h)=\sum h_{(1)}\uot h_{(2)},
$$ for any $m\in M$, $$\Delta(m)=(\delta_L+\delta_R)(m)=\sum m_{(-1)}\uot m_{(0)}+\sum m_{(0)}\uot m_{(1)},$$
and for any $m^1\sq \cdots \sq m^n\in M^{\sq _H n}$ with $n\geq 2$,\begin{eqnarray*}
\lefteqn{\Delta(m^1\sq \cdots \sq m^n)}\\[3pt]
&=& \sum m^1_{(-1)}\uot m^1_{(0)}\sq m^2\sq \cdots \sq m^n+\sum_{i=1}^{n-1}m^1\sq \cdots \sq m^i\uot  m^{i+1}\sq \cdots \sq m^n\\[3pt]
&&+ \sum m^1\sq \cdots \sq m^{n-1}\sq m^n_{(0)}\uot m^n_{(1)}.
\end{eqnarray*}We also define $\varepsilon:T^c_H(M)\rightarrow\gf $ by requiring $\varepsilon|_H=\varepsilon_H$ and $\varepsilon|_{M^{\sq _H
n}}=0$ for $n\geq 1$. Together with these maps, $T^c_H(M)$ is a
graded coalgebra called the cotensor coalgebra over $H$ and $M$,
which is characterized by the following universal property (cf.
\cite{N}).

\begin{proposition}[\cite{N}]\label{Prop:UP1}
Let $(C,\Delta_C)$ be a coalgebra and $g:C\rightarrow H$ be a
coalgebra map. We view $C$ as an $H$-bicomodule with left coaction
$(g\ot \id_C)\Delta_C$ and right coaction $(\id_C\ot g)\Delta_C$
respectively. If $f:C\rightarrow M$ is an $H$-bicomodule map such
that $f(\mathrm{corad}(C))=0$, then there exists a unique
coalgebra map $F:C\rightarrow T^c_H(M)$ making the following
diagrams commute:  \begin{displaymath}
\xymatrix{T^c_H(M)\ar[d]_\pi&C\ar@{.>}[l]_-F\ar[dl]^g\\
H&},
\xymatrix{T^c_H(M)\ar[d]_p&C\ar@{.>}[l]_-F\ar[dl]^f\\
M&},
\end{displaymath}where $\pi$ and $p$ are the projections from $T^c_H(M)$ onto $H$ and $M$ respectively.

Explicitly, $F=g+\sum_{n\geq 1}f^{\ot n}\circ \Delta_C^{(n-1)}$.
\end{proposition}

\subsubsection{Algebra structure on $T^c_H(M)$}
We will construct a bialgebra structure on $T^c_H(M)$ when $M$
admits moreover an algebra structure, generalizing constructions
in \cite{NR} and \cite{JRZ}. We mention that this bialgebra
appears recently as a special case in the framework of multi-brace
cotensor Hopf algebras and is used to realize the whole quantum
group (\cite{FR}). We would like to give a direct approach here.

Define
$$g=\mu_H (\pi \ot \pi):T^c_H(M)\uot T^c_H(M)\rightarrow H,$$ and
$$f=\cdot_L (\pi\ot p)+\cdot_R (p\ot \pi)+\mu (p\ot p):T^c_H(M)\uot
T^c_H(M)\rightarrow M,$$where $\cdot_L$ and $\cdot_R$ are the left
and right module structure maps of $M$ respectively.

Endowed with the coproduct of the tensor product of two
coalgebras, $T^c_H(M)\uot T^c_H(M)$ is a graded coalgebra. Since
both $\pi$ and $\mu_H$ are coalgebra maps, $g$ is a coalgebra map
as well. For an analogue reason, $f$ is an $H$-bicomodule map.

Since the coalgebra $(T^c_H(M), \Delta)$ is graded, we have that
$\mathrm{corad}(T^c_H(M))\subset \mathrm{corad}(H)\subset H$. On
the other hand, it is easy to show that for any two coalgebras $C$
and $D$ we have $\mathrm{corad}(C\otimes
D)\subset\mathrm{corad}(C)\otimes \mathrm{corad}(D)$, where
$C\otimes D$ is endowed with the tensor coproduct (cf. \cite{S}).
So $\mathrm{corad}(T^c_H(M)\uot T^c_H(M))\subset H\uot H$,  and
consequently, $f\big(\mathrm{corad}(T^c_H(M)\uot
T^c_H(M))\big)=0$.

By the universal property, there exists a unique coalgebra map
$F:T^c_H(M)\uot T^c_H(M)\rightarrow T^c_H(M)$ such that $\pi F=g$
and $p   F=f$. We want to show that $F$ is indeed an associative
product. In order to prove the associativity of $F$, we consider
$$\pi F(F\ot \id_{T^c_H(M)}),\ \ \pi F(\id_{T^c_H(M)}\ot
F):\ T^c_H(M)\uot T^c_H(M)\uot T^c_H(M)\rightarrow H,$$
and
$$pF(F\ot \id_{T^c_H(M)}),\ \ p F(\id_{T^c_H(M)}\ot F):\ T^c_H(M)\uot
T^c_H(M)\uot T^c_H(M)\rightarrow M.$$
It is clear that $\pi
F(F\ot \id_{T^c_H(M)})$ and $\pi F(\id_{T^c_H(M)}\ot F)$ are
coalgebra maps, while $p F(F\ot \id_{T^c_H(M)})$ and $p
F(\id_{T^c_H(M)}\ot F)$ are $H$-bimodule maps. If we could show that
$$\pi F(F\ot \id_{T^c_H(M)})=\pi F(\id_{T^c_H(M)}\ot F),\ \ \text{and}\ \ p F(F\ot \id_{T^c_H(M)})=p F(\id_{T^c_H(M)}\ot F),$$
then the associativity of $F$ is a consequence of the universal property of $T_H^c(M)$.

The first identity arises from the associativity of $\mu_H$. Let us prove the second one:
\begin{eqnarray*}
\lefteqn{p F (F\ot \id_{T^c_H(M)})}\\[3pt]
&=&\cdot_L(\mu_H(\pi\ot \pi)\ot p)+\cdot_R(f\ot \pi)+\mu(f\ot p)\\[3pt]
&=&\cdot_L(\pi\ot f)+\cdot_R(p\ot \mu_H(\pi\ot \pi))+\mu(p\ot f)\\[3pt]
&=&p  F (\id_{T^c_H(M)}\ot F).
\end{eqnarray*}

We denote $x\ast y=F(x\uot y)$ for any $x,y\in T^c_H(M)$. Then the
following theorem is proved in the above argument.

\begin{theorem}Let $M$ be an $H$-Hopf bimodule algebra. Then $(T^c_H(M),\ast,\Delta,\varepsilon)$ is a bialgebra.
\end{theorem}

\begin{remark}We remind the reader that the $H$-Hopf bimodule algebra $M$ is not necessarily unital in the above construction. \end{remark}

\subsection{The structure of $(T^c_H(M),\ast,\Delta,\varepsilon)$}

The bialgebra $(T^c_H(M),\ast,\Delta,\varepsilon)$ is constructed
by using the universal property in the last subsection. But the
algebra structure is not clear from the abstract definition. We
will show in this subsection that this product admits a
combinatorial interpretation by using the quantum quasi-shuffle
product (\cite{JRZ}).

\subsubsection{Quantum quasi-shuffle algebras}

For stating the construction of quantum quasi-shuffle algebras, we
introduce some necessary notions first.

\begin{definition}A quadruple $(A,m_A,\s)$ is called a
 \emph{braided algebra} if $(A,m_A)$ is an associative algebra with a braiding $\s$ on $A$ satisfying the following conditions:
$$(\id_A\ot m_A)\s_1\s_2=\s( m_A\ot \id_A),\ \ ( m_A\ot \id_A)\s_2\s_1=\s(\id_A\ot m_A).$$
A braided algebra $(A,m_A,\s)$ is called unital if there exists $1_A\in A$ such that $(A,m_A,1_A)$ is a unital associative algebra satisfying
$$\s(a\ot 1_A)=1_A\ot a,\ \ \s(1_A\ot a)=a\ot
1_A.$$

A quadruple $(C,\Delta_C,\varepsilon_C,\s)$ is called a
\emph{braided coalgebra} if $(C,\Delta_C,\varepsilon_C)$ is a
coalgebra with braiding $\s$ on $C$ satisfying the following
conditions:
$$
\s_1\s_2 (\Delta_C\ot  \id_C)= (\id_C\ot  \Delta_C)\s,\ \
\s_2\s _1(\id_C\ot\Delta_C)=(\Delta_C\ot \id_C)\s,$$
$$
(\id_C\ot
\varepsilon_C )\s=\varepsilon_C\ot \id_C,\ \
(\varepsilon_C \ot \id_C )\s =\id_C\ot  \varepsilon_C .
$$
A sextuple $(B,m_B,1_B,\Delta_B,\varepsilon_B,\s)$ is called a
\emph{braided bialgebra} if $(B,m_B,1_B,\s)$ is a braided algebra
and $(B,\Delta_B,\varepsilon_B,\s)$ is a braided coalgebra such
that
\[\left\{
\begin{split}\Delta_B m_B&=(m_B\ot m_B) (\id_B\ot
\s\ot\id_B)(\Delta_B\ot \Delta_B),\\[3pt]
\Delta_B(1_B)&=1_B\ot 1_B.\end{split} \right.
\]
We sometimes denote it by $(B,\s)$ for simplifying the
notation.\end{definition}

Since $\cyd$ is a braided tensor category, an algebra (resp.
coalgebra) object in $\cyd$ is naturally a braided algebra (resp.
coalgebra) with respect to the commutativity constraint in $\cyd$.

Let $\mathcal{BB}$ be the category of braided bialgebras
consisting of the following data:\\
(i) objects in $\mathcal{BB}$ are braided bialgebras;\\
(ii) morphisms between two objects $(B,\s)$ and $(B',\tau)$ in
$\mathcal{BB}$ are bialgebra morphisms $f:B\ra B'$ satisfying
$\tau  (f\ts f)=(f\ts f) \s$.

For an object $(B,\s)\in\mathcal{BB}$, $B\ts B$ is a braided
bialgebra with the product $(m_B\ot m_B) (\id_B\ot\s\ts\id_B)$ and
coproduct $(\id_B\ot\s\ot\id_B) (\Delta_B\ot \Delta_B)$.

We recall the definition of quantum quasi-shuffle algebras in a
general setting. Suppose that $(A,m,\s )$ is a braided algebra.
The quantum quasi-shuffle product $\qqs_\s$ on $T(V)$ is defined
as follows (\cite{JRZ}): for any $\lambda\in \gf$ and $x\in T(A)$,
the quantum quasi-shuffle product of $\lambda$ and $x$ is just the
scalar multiplication $\lambda x$. For $i,j\geq 2$ and any
$a_1,\ldots, a_i,b_1,\ldots, b_j\in A$, the product $\qqs_\s$ is
defined recursively by:\par $a_1\qqs_\s b_1= a_1\ot b_1+\s(a_1\ot
b_1)+m(a_1\ot b_1),$
\begin{eqnarray*}\lefteqn{ a_1\qqs_\s  (b_1\ot \cdots\ot
b_j)}\\[3pt]
&=&\left({\id}_A^{\ts j+1}+({\id}_A\ts \qqs_{\s(1,j-1)})(\beta_{1,1}\ts {\id}_A^{\ts j-1})+m\ts{\id}_A^{\ts j-1}\right)(a_1\ot b_1\ot \cdots\ot  b_j),
\end{eqnarray*}
\begin{eqnarray*}
\lefteqn{(a_1\ot \cdots\ot  a_i)\qqs_\s  b_1}\\[3pt]
&=&\left({\id}_A\ts\qqs_{\s(i-1,1)}+\beta_{i,1}+(m\ts{\id}_A^{\ts i-1})({\id}_A\ts \beta_{i-1,1})\right)(a_1\ts\cdots\ts a_i\ts b_1),
\end{eqnarray*}
\begin{eqnarray*}
\lefteqn{(a_1\ot \cdots\ot  a_i)\qqs_\s  (b_1\ot \cdots\ot  b_j)}\\[3pt]
&=&a_1\ot  \big((a_2\ot \cdots\ot  a_i)\qqs_\s  (b_1\ot \cdots\ot  b_{j})\big)\\[3pt]
&&+(\id_A\ot  \qqs_{\s  (i,j-1)})(\beta_{i,1}\ot \id_A^{\ot  j-1})(a_1\ot \cdots\ot  a_i\ot  b_1\ot \cdots\ot  b_j)\\[3pt]
&&+(m\ot \qqs_{\s  (i-1,j-1)} )(\id_A\ot  \beta_{i-1,1}\ot
\id_A^{\ot  j-1})(a_1\ot \cdots\ot  a_i\ot b_1\ot \cdots\ot  b_j),
\end{eqnarray*}
where $\qqs_{\s  (i,j)}$ denotes the restriction of $\qqs_\s $ on
$A^{\ot  i}\uot A^{\ot  j}$.

\begin{remark}
A new combinatorial description of the quantum quasi-shuffle product in the spirit of operads is discovered by the first author in \cite{F} by lifting shuffles to the generalized virtual braid monoid.
\end{remark}

It is shown in \cite{JRZ} that $(T(V),\qqs_\s)$ is a unital
algebra, which is called the quantum quasi-shuffle algebra
associated to the braided algebra $A$ and denoted by
$T_{\s,m}(A)$. Moreover, equipped with the deconcatenation
coproduct and the braiding $\beta$, $T_{\s,m}(A)$ is a braided
bialgebra. It has a universal property (\cite{JRZ}) which will
play an essential role in the later use. To state this
property, we need the following notion.

\begin{definition}Suppose $(C,\Delta_C, \varepsilon_C)$ is a coalgebra with a
preferred group-like element $1_C$. We denote
$\overline{\Delta_C}(x)=\Delta_C(x)-x\ot  1_C-1_C\ot  x$ for any
$x\in C$ and recursively
\[\begin{split}
F_0C&=\mathbb{K}1_C,\\[3pt]
F_rC&=\{x\in C|\overline{\Delta_C}(x)\in F_{r-1}C\ot F_{r-1}C \},\
\ \text{for }r\geq 1.
\end{split}\]
The coalgebra $C$ is said to be \emph{connected} if
$C=\bigcup_{r\geq 0}F_r C$.
\end{definition}

We denote by $\mathcal{CB}$ the full subcategory of $\mathcal{BB}$
whose objects are braided bialgebras $(B,\s)$ such that both $B$
and $B\ts B$ are connected.

Given a braided algebra $(A,m_A,\s)$, the quantum quasi-shuffle
algebra $T_{\s,m_A}(A)$ is uniquely determined by the following
universal property.
\begin{proposition}[\cite{JRZ}]

Let $(B,\tau)\in\mathcal{CB}$ be a braided bialgebra with
preferred group-like element $1_B$. Let $f:B\rightarrow A$ be a
linear map such that $m_A (f\ot f)=f  m_B$ on
$\ker\varepsilon_B\ot \ker\varepsilon_B$, $f(1_B)=0$ and $(f\ot f)
\tau=\s (f\ot f)$. Then there exists a unique braided bialgebra
morphism $\overline{f}:B\rightarrow T_{\s,m}(A)$ extending $f$.
Explicitly, $\overline{f}=\varepsilon_B+\sum_{n\geq 1}f^{\ot n}
\overline{\Delta_B}^{(n-1)}$.

\end{proposition}

\subsubsection{Radford biproduct construction}
In order to analyze the algebra structure of
$(T^c_H(M),\ast,\Delta)$, we need some results about bialgebras
with a projection onto a Hopf algebra due to Radford (\cite{Ra}).

Let $K$ be a Hopf algebra and $(A,m_A,1_A,\Delta_A,\varepsilon_A)$ be a bialgebra with two
bialgebra maps $i: K\rightarrow A$ and $\pi: A\rightarrow K$ such
that $\pi  i=\id_K$. Set $\Pi=\id_A\star(i  S  \pi)$, where
$\star$ is the convolution product on $\mathrm{End}(A)$, and
$B=\Pi(A)$. We give a summary on the main results in \cite{Ra}.
\begin{enumerate}
\item Equipped with left action $m_A (i\ot \id_A)$, right action
$m_A (\id_A\ot i)$, left coaction $(\pi\ot \id_A) \Delta_A$ and
right coaction $(\id_A\ot \pi) \Delta_A$, $A$ is a $K$-Hopf
bimodule. The set of right coinvariants $A^R$ is exactly $B$.
\item $B$ is a subalgebra of $A$. If we define $\Delta_B=(\Pi\ot
\Pi) \Delta_A$, then $(B,\Delta_B,\varepsilon_A|_B)$ is a
coalgebra.
\item Equipped with the standard $K$-Yetter-Drinfeld module
induced from the $K$-Hopf bimodule structure, $(B,m_A|_B,
1_A,\Delta_B, \varepsilon_A|_B)$ is a braided bialgebra.
\item The map $B\ot K\rightarrow A$ given by $b\ot h\mapsto bi(h)$
is a bialgebra isomorphism, where $B\ot K$ is endowed with the
smash product and smash coproduct.
\end{enumerate}

\subsubsection{The smash structure of $(T^c_H(M),\ast,\Delta)$}

Note that both of the inclusion map $i:H\rightarrow T^c_H(M)$ into
degree 0 and the projection map $\pi :T^c_H(M)\rightarrow H$ onto
degree 0 are bialgebra maps such that $\pi  i=\id_H$. By applying
the construction of Radford to $(T^c_H(M),\ast,1_H,\Delta,\varepsilon)$,
$(T^c_H(M)^R,\ast)$ is a subalgebra of
$(T^c_H(M),\ast)$, and $T^c_H(M)^R\# H$ is isomorphic, as a
bialgebra, to $T^c_H(M)$ via the isomorphism $\xi$ given by
$\xi(x\ot h)=x\ast i(h)$ for any $x\in T_H^c(M)^R$ and $h\in H$.
According to the preceding discussion, the investigation of
$(T^c_H(M),\ast)$ can be restricted to the subalgebra
$(T^c_H(M)^R,\ast)$.

We denote by $T(M^R)$ the tensor vector space on $M^R$.

\begin{proposition}The vector spaces
$T^c_H(M)^R$ and $T(M^R)$ are isomorphic as $H$-Yetter-Drinfeld
modules.
\end{proposition}
\begin{proof}
Essentially, this is proved in \cite{Ro}. For each $n\geq 1$, we
define
$$\varphi:(M^{\sq_H n})^R\rightarrow (M^R)^{\ot n}\ \ \text{and}\ \ \psi: (M^R)^{\ot n}\rightarrow (M^{\sq_H n})^R$$
as follows: for any $\sum m^1\sq\cdots\sq m^n\in M^{\sq_H n}$ and
$v^1\ot \cdots\ot v^n\in (M^{R})^{\ot n}$, $\varphi$ sends $\sum
m^1\sq\cdots\sq m^n$ to$$\sum m_{(0)}^1S(m_{(1)}^1)\ot \cdots\ot
m_{(0)}^{n-1}S(m_{(1)}^{n-1})\ot m^n_{(0)}S(m_{(1)}^n),$$ $\psi$
sends $v^1\ot \cdots\ot v^n$ to
$$\sum v^1 v_{(-1)}^2v_{(-2)}^3\cdots v_{(-n+1)}^n\sq\cdots\sq v_{(0)}^{n-1}v_{(-1)}^n\sq v_{(0)}^n.$$

Then it is routine to verify that both $\varphi$ and $\psi$ are
well-defined $H$-Yetter-Drinfeld module morphisms satisfying
$\varphi\circ\psi=\id_{(M^R)^{\ot n}}$ and
$\psi\circ\varphi=\id_{(M^{\sq_H n})^R}$.
\par
This proves the proposition since both sides are graded.
\end{proof}

We fix this isomorphism $\vp: T^c_H(M)^R\ra T(M^R)$ in the sequel.

Our next task is to show that the subalgebra $(T^c_H(M)^R,\ast)$
is isomorphic to the quantum quasi-shuffle algebra built on $M^R$.

\begin{theorem}
The map $\varphi: T^c_H(M)^R\rightarrow T_{\s,m}(M^R)$ is a morphism of braided bialgebras. As a consequence, $T^c_H(M)$ is
isomorphic, as a Hopf algebra, to $T_{\s,m}(M^R)\# H$.
\end{theorem}
\begin{proof}
Together with the coproduct $\delta=(\Pi\ot \Pi) \Delta$,
$T^c_H(M)^R$ is a braided bialgebra, where $\Pi$ is the
convolution product of $\id_{T^c_H(M)}$ and $i  S  \pi$. Note that
$T^c_H(M)^R$ and $T^c_H(M)^R\uot T^c_H(M)^R$ are connected since
\[
F_0T^c_H(M)^R=\gf 1_H ,\ \
F_rT^c_H(M)^R=\bigoplus_{0\leq i\leq r}\Pi(M^{\sq_H i}),\ \ r\geq 1
\]
and
\[\left\{
\begin{split}
F_0(T^c_H(M)^R\uot T^c_H(M)^R)&=\gf 1_H \uot 1_H,\\[3pt]
F_r(T^c_H(M)^R\uot T^c_H(M)^R)&=\bigoplus_{0\leq i_1+i_2\leq
r}\Pi(M^{\sq_H i_1})\uot \Pi(M^{\sq_H i_2}),\ \ r\geq 1.
\end{split} \right.
\]

Define a linear map $f:T_H^c(M)^R\ra M^R$ by $0$ on degree other
than $1$ and $f=P_R$ on degree $1$. It is obvious that $f$
commutes with the braidings and is an algebra map on
$\ker\varepsilon \ot \ker\varepsilon
=T^c_H(M)^R\uot T^c_H(M)^R$ with $f(1_H)=0$.
 According to Proposition 2.10, it extends to $\overline{f}:T_H^c(M)^R\ra
T_{\s,m}(M^R)$. Then it remains to show that
$\overline{f}=\varphi$ on each degree. This is direct since $f$
concentrate on degree $1$, the only non-zero term in the
restriction of $\overline{f}$ on $(M^{\sq_H n})^R$ is $f^{\ts
n}\circ\Delta^{(n-1)}$, which is given by the action of $P_R^{\ts
n}$.
\end{proof}

\section{Other structures and properties}
\subsection{Rota-Baxter algebra structure on $(T^c_H(M),\ast)$} We
turn to establish a Rota-Baxter algebra structure on
$(T_H^c(M),\ast)$. A basic reference is \cite{G}.

\begin{definition}Given $\lambda\in\gf$, A pair $(R,P)$ is called a \emph{Rota-Baxter algebra of weight $\lambda$} if $R$ is a $\gf$-algebra and $P$ is a $\gf$-linear endomorphism of $R$ satisfying $$P(x)P(y)=P(xP(y))+P(P(x)y)+\lambda P(xy),\forall x,y\in R.$$The map $P$ is call a \emph{Rota-Baxter operator} of weight $\lambda$.\end{definition}

If $P$ is a Rota-Baxter operator of weight 1, then $\lambda P$ is
a Rota-Baxter operator of weight $\lambda$ (cf. \cite{G}). So we
can just focus on the case of weight 1.

Given a Rota-Baxter algebra, the following remarkable property
enables one to provide new Rota-Baxter algebra structure on the
underlying space with a different product from the original one
and the Rota-Baxter operator.

\begin{lemma}[\cite{G}]
Let $(R,P)$ be a Rota-Baxter algebra of weight $\lambda$. We define $R_\heartsuit$ be the vector space $R$ with the multiplication
$$x\heartsuit y=xP(y)+P(x)y+\lambda xy,\ \ \text{for}\ x,y\in R.$$
Then $(R_\heartsuit,P)$ is again a Rota-Baxter algebra of weight
$\lambda$.
\end{lemma}

Now we recall a construction of Rota-Baxter algebras from quantum
quasi-shuffle algebras (cf. \cite{J}). Let $(A,m,\s)$ be a unital
braided algebra. We consider the vector space $A\uot T(A)$. As
vector spaces, $A\uot T(A)\cong T^+(A)=A\oplus A^{\ot 2}\oplus
\ldots$. We define an associative product
$$\lozenge_\s=(m\ot \qqs_\s)(\id_A\ot \beta \ot\id_{T(A)})$$
 on $A\uot T(A)$. We denote by $\mathcal{A}$ the algebra $(A\uot T(A),\lozenge_\s )$. We also define an endomorphism $Q$ on $A\uot T(A)$ by $Q(a\uot x)=1_A\uot a\ot x$. Then $(\mathcal{A}, Q)$ a Rota-Baxter
algebra of weight $1$.
\par
We let $T_{\s,m}^{+}(A)$ denote the sub-algebra of $T_{\s,m}(A)$ containing elements of positive degrees.

\begin{proposition}
We have that $\mathcal{A}_\heartsuit$ is isomorphic to
$T_{\s,m}^{+}(A)$ as an algebra. Therefore $(T_{\s,m}^{+}(A),P^+)$
is a Rota-Baxter algebra of weight $1$ with $P^+(x)=1_A\ts x$ for
$x\in T_{\s,m}^{+}(A)$.
\end{proposition}
\begin{proof} We define a linear map $f:A\uot T(A)\rightarrow
T^+(A)$ by $f(a\uot x)=a\ot x$ for any $a\uot x\in A\uot T(A)$. It
is apparent that $f$ is a bijection. Notice that for any $a,b\in
A$, $x\in A^{\ot i}$ and  $y\in A^{\ot j}$,\begin{eqnarray*}
\lefteqn{f((a\uot x)\heartsuit (b\uot y))}\\[3pt]
&=&f(Q+(a\uot x)\lozenge_\s (b\uot y)+(a\uot x)\lozenge_\s Q(b\uot y)+(a\uot x)\lozenge_\s (b\uot y))\\[3pt]
&=&f((1_A\uot a\ot x)\lozenge_\s (b\uot y) +(a\uot x)\lozenge_\s
(1_A\uot b\ot y)+(a\uot x)\lozenge_\s (b\uot y))\\[3pt]
&=&(m\ot \qqs_{\s(i+1,j)})(\id_A\ot \beta_{i+1,1} \ot\id_A^{\ot
j})(1_A\ot a\ot x\ot b\ot y)\\[3pt]
&&+(m\ot \qqs_{\s(i,j+1)})(\id_A\ot \beta_{i,1} \ot\id_A^{\ot
j+1})(a\ot x\ot 1_A\ot b\ot y)\\[3pt]
&&+(m\ot \qqs_{\s(i,1)})(\id_A\ot \beta_{i,1} \ot\id_A^{\ot
j})(a\ot x\ot b\ot y)\\[3pt]
&=&(\id_A\ot  \qqs_{\s  (i+1,j)})(\beta_{i+1,1}\ot \id_A^{\ot  j})(a\ot x\ot b\ot y)\\[3pt]
&&+a\ot  (x\qqs_\s  (b\ot y))\\[3pt]
&&+(m\ot \qqs_{\s  (i,j)} )(\id_A\ot  \beta_{i,1}\ot
\id_A^{\ot  j})(a\ot x\ot b\ot y)\\[3pt]
&=&(a\ot x)\qqs_\s (b\ot y)\\[3pt]
&=&f(a\uot x)\qqs_\s f(b\uot y).
\end{eqnarray*}
\end{proof}

We assume that $A$ is a braided algebra with unit $1_A$ and define
an endomorphism $P$ on $T_{\s,m}(A)$ as follows:\[\left\{
\begin{split}
P(\lambda)&=\lambda 1_A,\text{ if }\lambda\in \gf,\\[3pt]
P(x)&=P^+(x),\text{ if }x\in T^+(A).
\end{split}\right.
\]

\begin{proposition}\label{Prop:RB}
The pair $(T_{\s,m}(A),P)$ is a Rota-Baxter algebra of weight
$1$.\end{proposition}
\begin{proof}First of all, for any $\lambda,\nu\in \gf$,\begin{eqnarray*}
\lefteqn{P\big(P(\lambda)\qqs_\s \nu \big)+P\big(\lambda\qqs_\s P(\nu )\big)+P( \lambda\nu)}\\[3pt]
&=& \lambda \nu\big(1_A\ot 1_A+1_A\ot 1_A+m(1_A \ot 1_A)\big)\\[3pt]
&=&(\lambda 1_A)\qqs_\s (\nu 1_A)=P(\lambda)\qqs_\s P(\nu).
\end{eqnarray*}

For any $x\in T^+(A)$,
\begin{eqnarray*}
\lefteqn{P\big(P(\lambda)\qqs_\s x \big)+P\big(\lambda\qqs_\s P(x )\big)+P( \lambda\qqs_\s x)}\\[3pt]
&=& \lambda 1_A\ot( 1_A\qqs_\s x)+\lambda 1_A\ot 1_A\ot x +\lambda 1_A \ot x\\[3pt]
&=&(\lambda 1_A)\qqs_\s  (1_A\ot x)=P(\lambda)\qqs_\s P(x),
\end{eqnarray*}
and
\begin{eqnarray*}
\lefteqn{P\big(P(x)\qqs_\s \lambda \big)+P\big(x\qqs_\s P(\lambda )\big)+P( x\qqs_\s \lambda)}\\[3pt]
&=& \lambda 1_A\ot 1_A\ot x+\lambda 1_A\ot(x\qqs_\s  1_A) +\lambda 1_A \ot x\\[3pt]
&=&(1_A\ot x)\qqs_\s (\lambda 1_A)=P(x)\qqs_\s P(\lambda),
\end{eqnarray*}

For any $x,y\in T^+(A)$,\begin{eqnarray*}
\lefteqn{P\big(P(x)\qqs_\s y \big)+P\big(x\qqs_\s P(y )\big)+P( x\qqs_\s y)}\\[3pt]
&=&  1_A\ot\big( (1_A\ot x)\qqs_\s y\big)+ 1_A\ot\big(x\qqs_\s  (1_A\ot y)\big) + 1_A \ot( x\qqs_\s y)\\[3pt]
&=&(1_A\ot x)\qqs_\s ( 1_A\ot y)=P(x)\qqs_\s P(y).
\end{eqnarray*}
In the above computations, we used the fact that $\beta(1_A\uot
x)=x\uot 1_A$ and $\beta(x\uot 1_A)=1_A\uot x$ for any $x\in
T^+(A)$.

Finally,
\begin{eqnarray*}
\lefteqn{P\big(P(\lambda+x)\qqs_\s (\nu+y) \big)+P\big((\lambda+x)\qqs_\s P(\nu+y )\big)+P\big( (\lambda+x)\qqs_\s (\nu+y)\big)}\\[3pt]
&=&P\big(P(\lambda)\qqs_\s \nu+P(\lambda)\qqs_\s y+P(x)\qqs_\s
\nu+P(x)\qqs_\s y \big)\\[3pt]
&&+P\big(\lambda\qqs_\s P(\nu )+\lambda\qqs_\s P(y )+x\qqs_\s
P(\nu )+x\qqs_\s P(y
)\big)\\[3pt]
&&+P\big( \lambda\qqs_\s \nu+\lambda\qqs_\s y+x\qqs_\s
\nu+x\qqs_\s
y\big)\\[3pt]
&=&P(\lambda)\qqs_\s P(\nu)+P(\lambda)\qqs_\s P(y)+P(x)\qqs_\s P(\nu)+P(x)\qqs_\s P(y)\\[3pt]
&=&P(\lambda+x)\qqs_\s P(\nu+y).
\end{eqnarray*}\end{proof}


We introduce the notion of Rota-Baxter algebras in $\cyd$. A
Rota-Baxter algebra $(R,P)$ is said to be in $\cyd$ if $R$ is an
object in $\cyd$ and $P$ is a morphism in $\cyd$. There exists a
unique linear extension of $P$ to the smash product $\wt{P}:R\#H
\ra R\#H $ given by $\wt{P}(a\# h)=P(a)\# h$.

\begin{lemma}\label{Lem:Bosonization}
Suppose that $(V,P)$ is a Rota-Baxter operator of weight $1$ in $\cyd$. Then so is the pair $(V\# H, \wt{P})$.
\end{lemma}
\begin{proof}
For any $v,v'\in V$ and $h,h'\in H$,\begin{eqnarray*}
\lefteqn{\wt{P}\big(\wt{P}(v\# h)(v'\# h')\big)+\wt{P}\big((v\# h)\wt{P}(v'\# h')\big)+\wt{P}\big((v\# h)(v'\# h')\big)}\\[3pt]
&=&\wt{P}\big((P(v)\# h)(v'\# h')\big)+\wt{P}\big((v\# h)(P(v')\# h')\big)+\wt{P}\big((v\# h)(v'\# h')\big)\\[3pt]
&=&\sum \left(P(v)(h_{(1)}\cdot v')\ot h_{(2)}h'+v\big(h_{(1)}\cdot P(v')\big)\ot h_{(2)}h'+ P\big(v(h_{(1)}\cdot v')\big)\ot h_{(2)}h'\right)\\[3pt]
&=&\sum \Big(P(v)(h_{(1)}\cdot v')+v P(h_{(1)}\cdot v')+P\big(v(h_{(1)}\cdot v')\big)\Big)\ot h_{(2)}h'\\[3pt]
&=&\sum P(v)\big(h_{(1)}\cdot P(v')\big)\ot h_{(2)}h'\\[3pt]
&=&(P(v)\# h)(P(v')\# h') \\[3pt]
&=&\wt{P}(v\# h)\wt{P}(v'\# h').
\end{eqnarray*}\end{proof}

Return to quasi-shuffle algebras: we take $V=T_{\s,m}(M^{R})$ in
the above lemma and $P$ the endomorphism of $T_{\s,m}(M^{R})$
defined before Proposition \ref{Prop:RB}. We define again by
$\wt{P}$ the following composition
$$\xymatrix{
T_H^c(M)\ar[r]^-{\phi\circ\xi} & T(M^R)\#H \ar[r]^-{\wt{P}} &
T(M^R)\#H \ar[r]^-{(\phi\circ\xi)^{-1}} & T_H^c(M)},$$where
$\phi=\varphi\ot \id_H$.

The following theorem is a corollary of Proposition \ref{Prop:RB}
and Lemma \ref{Lem:Bosonization}.

\begin{theorem}
Suppose $M$ is a unital $H$-Hopf bimodule algebra. Then the pair
$(T_H^c(M),\ast, \wt{P})$ is a Rota-Baxter algebra of weight $1$.
\end{theorem}

\begin{remark}By combining Lemma 3.2 and the above theorem, we can construct a new product on $T_H^c(M)$ which brings $T_H^c(M)$ another Rota-Baxter algebra structure with the same $\wt{P}$.\end{remark}

\subsection{Universal property} We study the universal property
of $(T_H^c(M),\ast)$. For this reason, it should start by
describing the suitable category in which $(T_H^c(M),\ast)$ is a
free object.

Let $(B,\mu_B,\Delta_B,\varepsilon_B)$ be a bialgebra such that:
\begin{enumerate}
\item $\mathrm{corad}(B)$ is a subalgebra of $B$;
\item there exist two bialgebra maps $g:B\rightarrow H$ and $k:H\rightarrow B$ with $g\circ k=\id_H$.
\end{enumerate}

If this is the case, we can endow $B$ with an $H$-Hopf bimodule algebra
structure with the help of these maps. As before, the left and right comodule structures of $B$ are given by $\varrho_L=(g\ot \id_B)\Delta_B$ and $\varrho_R=(\id_B\ot g)\Delta_B$ respectively. The $H$-bimodule structure is given by $h\cdot b\cdot h'=k(h)bk(h')$ for $h,h'\in
H$ and $b\in B$. Since both $g$ and $k$ are Hopf algebra maps with $g\circ k={\id}_H$, easy verifications show that
$\varrho_L(h\cdot b\cdot h')=\Delta_B(h)\varrho_L(b)\Delta_B(h')$
and $\varrho_R(h\cdot b\cdot
h')=\Delta_B(h)\varrho_R(b)\Delta_B(h')$. Since $\mu_B((b\cdot
h)\otimes b')=bk(h)b'=\mu_B(b\otimes (h\cdot b'))$, $\mu_B$
induces a multiplication $B\otimes _H B\rightarrow B$. In addition, by a
direct verification, $\mu_B$ is a morphism in $\chb$.

On the other hand, the condition $g\circ k=\id_H$ implies
$B=\mathrm{Im}k\oplus \mathrm{Ker} g$ as vector space.

\begin{theorem}[Universal property of $T_H^c(M)$]Under the assumptions above, for any $H$-bimodule algebra $M$ and any morphism $f:B\rightarrow M$ in $\chb$ such
that $f\circ k=0$, $f$ is an algebra map and
$f(\mathrm{corad}(B))=0$, there exists a unique
bialgebra map $F:B\rightarrow T^c_H(M)$ making the following
diagrams commute:
\begin{displaymath}
\xymatrix{T^c_H(M)\ar[d]_\pi&B\ar@{.>}[l]_-F\ar[dl]^g\\
H&},\ \ \ 
\xymatrix{T^c_H(M)\ar[d]_p&B\ar@{.>}[l]_-F\ar[dl]^f\\
M&},
\end{displaymath}where $\pi$ and $p$ are the projections from $T^c_H(M)$ onto $H$ and $M$ respectively.\end{theorem}
\begin{proof}By Proposition \ref{Prop:UP1}, there exists a unique coalgebra map $F:B\rightarrow T^c_H(M)$ making the above diagrams commute. We want to show that $F$ is an algebra map, i.e., $F(bb')=F(b)\ast F(b')$ whenever $b,b'\in B$. Define two maps $F_1,F_2: B\otimes B\rightarrow T^c_H(M)$ by $F_1(b\otimes b')=F(bb')$ and $F_2(b\otimes b')=F(b)\ast F(b')$
respectively. Since $F_1=F\mu_B$ and both $F$ and $\mu_B$ are
coalgebra maps, $F_1$ is a coalgebra map. So is $F_2$ for a
similar reason.

Consider the following commutative diagrams:
\begin{displaymath}
\xymatrix{T^c_H(M)\ar[d]_\pi&B\otimes B\ar[l]_-{F_1}\ar[dl]^{\pi  F_1}\\
H&},\ \ \  \xymatrix{T^c_H(M)\ar[d]_p&B\otimes
B\ar[l]_-{F_1}\ar[dl]^{p F_1}\\
M&},
\end{displaymath}

Since $\pi  F_1=\pi  F \mu_B= g \mu_B$ is the composition of two
coalgebra maps $g$ and $\mu_B$, $\pi  F_1$ is a coalgebra map.
Since $p F_1=p F \mu_B= f \mu_B$ is the composition of two
$H$-bicomodule morphisms $f$ and $\mu_B$, $p\circ F_2$ is a
$H$-bicomodule morphism. Note that\begin{eqnarray*}
p F_1(\mathrm{corad}(B\otimes B))&=&f \mu_B(\mathrm{corad}(B\otimes B))\\[3pt]
&\subset&f \mu_B(\mathrm{corad}(B)\otimes \mathrm{corad}(B))\\[3pt]
&=&f(\mathrm{corad}(B))=0.
\end{eqnarray*}
Hence, $F_1$ is the coalgebra map induced by $\pi F_1$ and $p F_1$
from Proposition \ref{Prop:UP1}.

Now consider another two commutative diagrams:
\begin{displaymath}
\xymatrix{T^c_H(M)\ar[d]_\pi&B\otimes B\ar[l]_-{F_2}\ar[dl]^{\pi  F_2}\\
H&},\ \ \ \xymatrix{T^c_H(M)\ar[d]_p&B\otimes
B\ar[l]_-{F_2}\ar[dl]^{p F_2}\\
M&},
\end{displaymath}

We have $\pi  F_2=\pi \circ\ast\circ (F\otimes F)=\mu_H (\pi\otimes
\pi)(F\otimes F)=\mu_H(g\otimes g)$, and $p F_2=\cdot_L(g\otimes
f)+\cdot_R (f\otimes g)+\mu (f\otimes f)$, where $\mu$ is the
multiplication of $M$. Evidently, $\pi  F_2$ is a coalgebra map
and $p  F_2$ is an $H$-bicomodule morphism. Note that $p
F_2(\mathrm{corad}(B\otimes B))\subset p
F_2(\mathrm{corad}(B)\otimes \mathrm{corad}(B))=(\cdot_L(g\otimes
f)+\cdot_R(f\otimes g)+\mu (f\otimes f))(\mathrm{corad}(B)\otimes
\mathrm{corad}(B))=0$. Hence, $F_2$ is the coalgebra map induced
by $\pi F_2$ and $p F_2$ from Proposition \ref{Prop:UP1}.

The last step is to show that $\pi  F_1=\pi  F_2$ and $p
 F_1=p  F_2$. Observe that for any $b,b'\in B$,
$$\pi F_1(b\otimes b')=g(bb')=g(b)g(b')=\pi F_2(b\otimes b').$$
By the decomposition $B=\mathrm{Im}k\oplus \mathrm{Ker} g$, we
write $b=k(h)+d$ and $b'=k(h')+d'$ for some $h,h'\in H$ and
$d,d\in \mathrm{Ker} g$. Then\begin{eqnarray*}\lefteqn{ p
F_2(b\otimes b')}\\[3pt]
&=&h \cdot f(k(h'))+g(d)\cdot f(k(h'))+ h\cdot f(d') +g(d)\cdot f( d'))\\[3pt]
&&+f(k(h))\cdot h'+ f(d)\cdot h'+ f(k(h))\cdot g( d') +f(d)\cdot g( d')\\[3pt]
&&+f(k(h))f( k(h'))+f(d)f( k(h'))+ f(k(h))f( d') +f(d)f( d')\\[3pt]
&=&f(h\cdot d')+ f(d\cdot h')+f(d)f( d')\\[3pt]
&=&f(bb')\\[3pt]
&=&p  F_1(b\otimes b'),
\end{eqnarray*}
where the second equality follows from the fact that $f$ is an
$H$-bimodule map and $f k=0$. This completes our proof.\end{proof}

\section{Examples}In this section, we will provide some concrete
examples of cofree Hopf algebras on Hopf bimodule algebras,
including universal Clifford algebras and universal quantum
groups.

\subsection{Universal Clifford algebras}

Let $V$ be a vector space, $Q:V\ra\mathbb{K}$ be a quadratic form
on $V$ and $<\cdot,\cdot>$ be the associated symmetric bilinear form
via polarization. The Clifford algebra $Cl(V,Q)$ on $V$ is the
associated $\mathbb{K}$-algebra generated by $V$ with respect to
the relations $uv+vu=<u,v>.1$ for any $u,v\in V$.
\par
Let $v_1,\cdots,v_n$ be a linear basis of $V$. The algebra
$Cl(V,Q)$ has generators $v_1,\cdots,v_n$ and relations
$v_iv_j+v_jv_i=<v_i,v_j>.1$.

\begin{definition}
The universal Clifford algebra $Cl(V)$ is generated as a
$\mathbb{K}$-algebra by $\epsilon$, $v_i$, $\xi_{ij}$ for $1\leq
i\leq j\leq n$ and relations:
$$\epsilon^2=1,\ \ \epsilon\xi_{ij}\epsilon^{-1}=\xi_{ij},\ \ \epsilon v_i\epsilon^{-1}=-v_i,$$
$$v_iv_j+v_jv_i=\xi_{ij},\ \ \xi_{ij}\xi_{kl}=\xi_{kl}\xi_{ij},$$
where $1\leq i\leq j\leq n$ and $1\leq k\leq l\leq n$.
\end{definition}

Let $I$ be the ideal in $Cl(V)$ generated by
$\xi_{ij}-Q(v_i,v_j).1$ for $1\leq i\leq j\leq n$. Then the
Clifford algebra $Cl(V,Q)$ is a subalgebra of the quotient algebra
$Cl(V)/I$ generated by $v_i$ for $i=1,\cdots,n$.
\par
The universal Clifford algebra admits a unique Hopf algebra structure defined by:
$$\Delta(\epsilon)=\epsilon\ts \epsilon,\ \ \Delta(v_i)=\epsilon\ts v_i+v_i\ts 1,\ \ \Delta(\xi_{ij})=1\ts \xi_{ij}+\xi_{ij}\ts 1,$$
$$\ve(\epsilon)=1,\ \ \ve(v_i)=0,\ \ \ve(\xi_{ij})=0.$$
Since
$$\Delta(\xi_{ij}-Q(v_i,v_j).1)=1\ts (\xi_{ij}-Q(v_i,v_j).1)+(\xi_{ij}-Q(v_i,v_j).1)\ts 1+Q(v_i,v_j).1\ts 1$$
and $Q$ is symmetric, the ideal $I$ is a Hopf ideal if and only if $Q(v_i,v_j)=0$ for any $i,j=1,\cdots,n$. If this is the case, the quotient $Cl(V)/I$ is isomorphic to the exterior algebra $\bigwedge(V)$ on $V$.
\par
We want to construct the universal Clifford algebra in the
framework of cofree Hopf algebras on Hopf bimodule algebras.
\par
Let $H=\mathbb{K}[\mathbb{Z}/2]= \mathbb{K}.1\oplus
\mathbb{K}.\epsilon$ be the group algebra. It admits a Hopf
algebra structure by letting $1$ and $\epsilon$ be group-like. Let
$W=U\ts H$ where $U$ is the vector space generated by $v_i$ and
$\xi_{ij}$ for $1\leq i\leq j\leq n$. It admits an $H$-Hopf
bimodule algebra structure as follows:
\begin{enumerate}
\item $U$ admits a left $H$-module and comodule strucure: for $1\leq i\leq j\leq n$,
$$\delta_L(v_i)=\epsilon\ts v_i,\ \ \delta_L(\xi_{ij})=1\ts\xi_{ij},\ \ \epsilon.v_i=-v_i,\ \ \epsilon.\xi_{ij}=\xi_{ij}.$$
\item Left $H$-module and comodule structures: for $u\in U$ and
$h,h'\in H$,
$$\delta_L(u\ts h)=\sum u_{(-1)}h_{(1)}\ts u_{(0)}\ts h_{(2)},\ \ h'.(u\ts h)=\sum h_{(1)}'.u\ts h_{(2)}'h.$$
\item Right $H$-module and comodule structures: for $u\in U$ and
$h,h'\in H$,
$$\delta_R(u\ts h)=\sum u\ts h_{(1)}\ts h_{(2)},\ \ (u\ts h)h'=u\ts hh'.$$
\item The multiplication $m:W\ts W\ra W$ is uniquely determined by: for $1\leq i\leq j\leq n$,
$$m(v_i\ts v_jh')=\xi_{ij}h',\ \ m(v_i\epsilon\ts v_jh')=-\xi_{ij}\epsilon h',$$
and on all other elements, $m$ gives zero. Notice that we omitted the tensor product inside of $W$ for simplification.
\end{enumerate}

We let $T_H^c(W)$ denote the cofree Hopf algebra built on the H-Hopf bimodule algebra $W$ and $Q_H(W)$ its subalgebra generated by $H$ and $W$. The multiplication in $Q_H(W)$ is denoted by $\ast$.

\begin{lemma}
There exists a Hopf algebra morphism $\varphi:Cl(V)\ra Q_H(W)$ defined by
$$v_i\mapsto v_i,\ \ \epsilon\mapsto\epsilon,\ \ \xi_{ij}\mapsto \xi_{ij}.$$
\end{lemma}

\begin{proof}
The only thing which needs a proof is that in $Q_H(W)$, the following identity holds: for any $1\leq i\leq j\leq n$, $v_i\ast v_j+v_j\ast v_i=\xi_{ij}$. Indeed, $$v_i\ast v_j=v_i\ts v_j-v_j\ts v_i+m(v_i\ts v_j)=v_i\ts v_j-v_j\ts v_i+\xi_{ij},$$
$$v_j\ast v_i=v_j\ts v_i-v_i\ts v_j.$$
\end{proof}

\begin{theorem}
$\varphi$ is an isomorphism of Hopf algebras.
\end{theorem}

\begin{proof}
Since $Q_H(W)$ is generated by $H$ and $W$, $\varphi$ is surjective.
\par
We show the injectivity: firstly we take the quotient of $Cl(V)$ by the Hopf ideal generated by $\xi_{ij}$ for $1\leq i\leq j\leq n$. This ideal is sent by $\varphi$ to a Hopf ideal $J$ in $Q_H(W)$ generated by $\xi_{ij}$ for $1\leq i\leq j\leq n$. Then $\varphi$ passes to the quotient to give a Hopf algebra surjection $\overline{\varphi}:\bigwedge(V)\ra Q_H(W)/J$ which is moreover an isomorphism.
\par
Therefore $\ker\varphi$ is included in the subalgebra $C$ of $Cl(V)$ generated by $\xi_{ij}$ for $1\leq i\leq j\leq n$, which is a polynomial algebra. Since the restriction of the multiplication $m$ on $W$ to the subspace spanned by these $\xi_{ij}$ is zero, the subalgebra of $Q_H(W)$ generated by $\xi_{ij}$ is a symmetric algebra. As a conclusion, the restriction of $\varphi$ to $C$ gives a Hopf algebra surjection
$$\vp|_C:C\ra\mathbb{K}[\xi_{ij}|\ 1\leq i\leq j\leq n],$$
hence $\vp|_C$ is injective and $\ker\varphi=\{0\}$.
\end{proof}

\subsection{Universal quantum groups}

We first recall the construction (with slight generalizations and modifications) in \cite{FR}.
\par
Let $\mathfrak{g}$ be a Kac-Moody Lie algebra associated to an integral matrix $C=(c_{ij})_{n\times n}\in M_n(\mathbb{Z})$ and $q\in\mathbb{K}^*$ not be $0$ and $1$. We let $I$ denote the index set $\{1,\cdots,n\}$.
\par
Let $H=\mathbb{K}[K_1^{\pm 1},\cdots,K_n^{\pm 1}]$ be the group algebra of the additive group $\mathbb{Z}^n$: it is a Hopf algebra, $W$ be the vector space spanned by $\{E_i,F_i,\xi_i|\ i\in I\}$ and $M=W\ts H$.
\par
The following abuse of notation will be applied hereafter: for $x\in W$ and $K\in H$, we shall write $x$ for $x\ts 1$ in $M$, and $xK$ for $x\ts K$ in $M$ for short.
\par
We define an $H$-Hopf bimodule algebra structure on $M$:
\begin{enumerate}
\item The right module and comodule structures are defined as follows: for $w\in W$ and $h,h'\in H$,
$$\delta_R(w\ts h)=\sum w\ts h_{(1)}\ts h_{(2)},\ \ (w\ts h)h'=w\ts hh'.$$

\item The left structures are given in the following way: on $W$, for any $i,j\in I$,
$$K_i.E_j=q^{c_{ij}}E_j,\ \ K_i.F_j=q^{-c_{ij}}F_j,\ \ K_i.\xi_j=\xi_j;$$
$$\delta_L(E_i)=K_i\ts E_i,\ \ \delta_L(F_i)=K_i\ts F_i,\ \ \delta_L(\xi_i)=K_i^2\ts \xi_i,$$
then the structure on $M=W\ts H$ is taken to be the one arising from the tensor product.
\item We define $m:M\ts M\ra M$ by: for any $K,K'\in H$, if $\lambda$ is the constant such that $K.F_j=\lambda F_j$,
$$m(E_iK\ts F_jK')=\delta_{ij}\lambda\xi_iKK';$$
and on any other elements not of the above form, $\alpha$ gives $0$.
\end{enumerate}

It is not difficult to show that $M$, with these structures, is an $H$-Hopf bimodule algebra (for example, see \cite{FR}).
\par
We let $T_H^c(M)$ denote the cofree Hopf algebra built on the $H$-Hopf bimodule algebra $M$ and $Q_H(M)$ be its sub-Hopf algebra generated by $H$ and $M$ whose multiplication will be denoted by $\ast$.

\begin{definition}
The Hopf algebra $Q_H(M)$ is called the universal quantum group associated to the matrix $C$.
\end{definition}

Some specializations of this universal object are of great interests:

\begin{enumerate}
\item Let $A$ be a symmetrizable Cartan matrix, $C=DA$ be its
symmetrization with $D=\text{diag}(d_1,\cdots,d_n)$ and
$\mathfrak{g}$ be the associated Kac-Moody Lie algebra. Suppose
that $q\in\mathbb{K}^*$ is not a root of unity. Let $J$ be the
ideal in $Q_H(M)$ generated by
$$\left\{\left.\xi_i-\frac{K_i^2-1}{q_i-q_i^{-1}}\right|\ i\in I\right\}\ \  \text{where}\ \  q_i=q^{d_i}.$$
The ideal $J$ is in fact a Hopf ideal and the following theorem is proved in \cite{FR}:
\begin{theorem}[\cite{FR}]
There exists an isomorphism of Hopf algebras $U_q(\mathfrak{g})\cong Q_H(M)/J$.
\end{theorem}
\item We respect all situations in (1) but suppose that $q^l=1$ is
a primitive odd root of unity for $l\geq d_i$ for any $i\in I$.
Let $I$ be the Hopf ideal in $Q_H(M)$ generated by $J$ and
$K_i^l-1$ for $i\in I$. Using Theorem 15 in \cite{Ro} and applying
the proof of the theorem above in \cite{FR} one has the following
corollary:
\begin{corollary}
Let $u_q(\mathfrak{g})$ be the small quantum group associated to $\mathfrak{g}$ (Frobenius
kernel). There exists an isomorphism of Hopf algebra
$u_q(\mathfrak{g})\cong Q_H(M)/I$.
\end{corollary}
\item If there is no restriction on the matrix $C\in M_n(\mathbb{Z})$ but $q$ is not a root of unity. Similar to the proof of Proposition 7 in \cite{FR}, it can be shown that the ideal generated by $\xi_i-P(K_i)$ for $i\in I$ and $P_i\in\mathbb{K}[K_1,\cdots,K_n]$ is a Hopf ideal if and only if there exist $\lambda_i\in\mathbb{K}$ such that $P(K_i)=\lambda_i(K_i^2-1)$. We suppose moreover that $\lambda_i\neq 0$ for any $i\in I$, then without loss of generality, we can suppose $\lambda_i=1$.
\par
The discussion above allows us to define the quantum group $U_q(C)$ associated to a general integral matrix $C$ as the quotient of the universal quantum group $Q_H(M)$ by the Hopf ideal generated by $\xi_i-K_i^2+1$ for $i\in I$.
\end{enumerate}

\section*{Acknowledgements}
This work was partially supported by the National Natural Science
Foundation of China (Grant No. 11201067).

\bibliographystyle{amsplain}

\end{document}